\documentclass[12pt]{article}
\usepackage{e-jc}
\usepackage{amsthm,amsmath,amssymb}
\usepackage{graphicx}
\usepackage[]{color}
\usepackage{cite}
\theoremstyle{plain}
\newtheorem{theorem}{Theorem}

\newtheorem{fact}[theorem]{Fact}
\newtheorem{corollary}[theorem]{Corollary}
\newtheorem{proposition}[theorem]{Proposition}

\theoremstyle{definition}

\newtheorem{conjecture}[theorem]{Conjecture}

\theoremstyle{remark}
\newtheorem{remark}[theorem]{Remark}

\def\Sep{{\sf Sep}}
\def\Sepp{{\sf Sepp}}
\def\lcm{{\sf lcm}}
\def\cA{\mathcal{A}}

\title{Lower Bounds on Words Separation:\\ Are There Short Identities in Transformation Semigroups?}

\author{Andrei A. Bulatov\thanks{Supported by an NSERC Discovery grant}\\
\small School of Computing Science\\[-0.8ex]
\small Simon Fraser University\\[-0.8ex] 
\small Burnaby, BC, Canada\\
\small\tt andrei.bulatov@gmail.com\\
\and
Olga Karpova \\
\small Institute of Mathematics\\[-0.8ex] \small and Computer Science\\[-0.8ex]
\small Ural Federal University\\[-0.8ex]
\small Ekaterinburg, Russia\\
\small\tt sckleppi@gmail.com\\
\and
Arseny M. Shur\thanks{Partialy supported by the grant 16-01-00795 of the Russian Foundation for Basic Research}\\
\small Institute of Mathematics\\[-0.8ex] \small and Computer Science\\[-0.8ex]
\small Ural Federal University\\[-0.8ex]
\small Ekaterinburg, Russia\\
\small\tt arseny.shur@urfu.ru\\
\and
Konstantin Startsev\\
\small Institute of Mathematics\\[-0.8ex] \small and Computer Science\\[-0.8ex]
\small Ural Federal University\\[-0.8ex]
\small Ekaterinburg, Russia\\
\small\tt kon7075@yandex.ru
}

\date{\dateline{Sep 10, 2016}{XX}}

\begin{document}

\maketitle

\begin{abstract}
The words separation problem, originally formulated by Goralcik and Koubek (1986), is stated as follows. Let $\Sep(n)$ be the minimum number such that for any two words of length $\le n$ there is a deterministic finite automaton with $\Sep(n)$ states, accepting exactly one of them. The problem is to find the asymptotics of the function $\Sep$. This problem is inverse to finding the asymptotics of the length of the shortest identity in full transformation semigroups $T_k$. The known lower bound on $\Sep$ stems from the unary identity in $T_k$. We find the first series of identities in $T_k$ which are shorter than the corresponding unary identity for infinitely many values of $k$, and thus slightly improve the lower bound on $\Sep(n)$. Then we present some short positive identities in symmetric groups, improving the lower bound on separating words by permutational automata by a multiplicative constant. Finally, we present the results of computer search for short identities for small $k$.

\end{abstract}

\section{Introduction}

Telling two inputs apart is one of the simplest computational problems one can imagine. As usual, the inputs are thought of as two  finite words $u,v$ 
over a finite alphabet $\Sigma$. Both $u$ and $v$ are known in advance; then one of them is fed to the algorithm which should decide whether this is $u$ or $v$. For a powerful computational model, such as the RAM model, the problem can be solved with constant space (in the length of the words): we need just one register to scan the input word until we reach a position in which $u$ and $v$ differ and look at the symbol at this position to decide whether we see $u$ or $v$ (a word can be supposed to end with a unique sentinel symbol). However, if the computational model is weak, like the finite automaton, the situation changes drastically, and distinguishing two words can no longer be done with constant space. The problem of determining the minimal size of a finite automaton separating two given words is NP-hard, as follows from some known algebraic results (see the discussion below). Moreover, even if we look at the maximal possible size of such automaton for words of a given length, very little is known about the asymptotics of this value. To make it more precise, we need some definitions. 

We use the array notation $w=w[1..n]$ to represent finite words over finite alphabet $\Sigma$ when appropriate, and also the standard notions of factors, prefixes, suffixes. We write $|w|$ for the length of $w$ and $|w|_x$ for the number of occurrences of the letter $x$ in $w$. We treat a deterministic finite automaton (dfa) as a quadruple $\cA=\{\Sigma,Q,\delta,s\}$, consisting of a finite alphabet, a finite set of states, a transition function, and an initial state. We write $q.w$ for the state of $\cA$ obtained by reading the word $w\in\Sigma^*$ starting in the state $q\in Q$. The dfa $\cA$ \emph{separates} words $u,v\in\Sigma^*$ if $s.u\ne s.v$. 
(Equivalently, there exists a set $T\subset Q$ of accepting states such that exactly one of the words $u,v$ is accepted.) Let $\Sep(u,v)$ be the minimum number of states in a dfa separating $u$ and $v$. 

Let $T_k$ denote the semigroup of all selfmaps of the set $\{1,\ldots, k\}$ under the composition of maps; it is called the \emph{full transformation semigroup} on $k$ elements. An \emph{identity} in a semigroup $T$ is a pair of words $(u,v)$ such that the images of $u$ and $v$ under any map $\Sigma\to T$ are equal as the elements of $T$. By the \emph{length} of the identity $(u,v)$ we mean the maximum of $|u|,|v|$. We write $u\equiv_k v$ to indicate the fact that $(u,v)$ is an identity in $T_k$. The \emph{transition semigroup} of a dfa $\cA$ is a subsemigroup of $T_{|Q|}$ consisting of all maps $w: q\to q.w$, where $w\in\Sigma^*$. The following simple fact connects identities and separation:

\begin{fact} \label{sep_ide}
For any words $u,v$, the identity $u\equiv_k v$ holds if and only if $\Sep(u,v)>k$. 
\end{fact}

Indeed, if $u\equiv_k v$, then this identity holds for the transition semigroup of any $k$-state dfa $\cA$, implying $q.u=q.v$ in it for any state $q$. If otherwise $\rho(u)\ne\rho(v)$ in $T_k$ for some map $\rho: \Sigma\to T_k$, then the transformations $\rho(a)$, $a\in\Sigma$ can be used to define transitions in the $k$-state dfa separating $u$ and $v$.

It is known that the problem of checking whether $u\equiv_k v$ is coNP-complete for any $k>2$ \cite{AVG09,Kli12}. So by Fact~\ref{sep_ide}, it is NP-complete to check whether $\Sep(u,v)\le k$.

Let $\Sep(n)=\max_{u,v\in\Sigma^{\le n}}\Sep(u,v)$. The problem of describing the asymptotics of $\Sep(n)$ was first posed by Goralcik and Koubek \cite{GoKo86}. Due to Fact~\ref{sep_ide}, this problem is equivalent to finding the asymptotics of the minimum length of an identity in $T_k$. For the existing results on the identities in $T_k$ see, e.g., \cite{PSSSV94} and the references therein. Up to now the shortest known identity in $T_k$ has been the unary identity
\begin{equation} \label{unary}
x^{k-1}=x^{k-1+\lcm(k)},
\end{equation}
where $\lcm(k)$ denotes the least common multiple of the integers $1,\ldots,k$. Hence, $\Sep(n)>k$ for $n\ge \lcm(k)+k-1$. Since $\log(\lcm(k))=k+o(k)$ by the Prime Number Theorem\footnote{In this paper, (a) the notation $\log$ stands for the natural logarithm; (b) the small-$o$-expressions can have any sign, so we always write '$+$' before them.}, this inequality can be rewritten as $\Sep(n)\ge \log n+o(\log(n))$. The logarithmic lower bound was presented already in \cite{GoKo86}, while the best known upper bound for $\Sep(n)$, obtained by Robson \cite{Rob89}, is $O(n^{2/5}\log^{3/5} n)$. Such a huge gap suggests that any of these bounds can be very loose. In this paper we present a new series of identities in $T_k$. These identities are shorter than \eqref{unary} whenever $k$ is a prime or a power of an odd prime. (More precisely, if $k=p^i$ for a prime $p$, then our identities are approximately $p/2$ times shorter than \eqref{unary}.) As far as we know, this is the first example of identities in $T_k$ that are shorter than \eqref{unary}.

There are several variations of the words separation problem; see, e.g., \cite{DESW11}. One variation requires a separating dfa to be \emph{permutational}, which means that every letter acts on the set of states as a permutation (i.e., $|Q.a|=|Q|$ for any $a\in\Sigma$). We denote the analog of the function $\Sep$ for permutational automata by $\Sepp$. Similar to Fact~\ref{sep_ide}, $\Sepp(u,v)>k$ if and only if the pair $(u,v)$ is an identity of the symmetric group $S_k$. Such group identities in semigroup signature are called \emph{positive} and denoted below by $u \cong_k v$. The best known upper bound for $\Sepp(n)$ also belongs to Robson \cite{Rob96} and is $O(n^{1/2})$. To get reasonable lower bounds on $\Sepp(n)$, one should find positive identities in $S_k$ which are shorter than the unary identity $x^{\lcm(k)}=1$. In general, the problem of finding short identities in finite symmetric groups has drawn some attention in the literature. The existence of an identity of length $O(e^{\sqrt{n\log n}})$ was proved in \cite{BoMc11} based on Landau's bound on the maximum order of a permutation \cite{Lan03}. Very recently, the existence of identities of length $O(e^{\log^4 n \log\log n})$ was established by Kozma and Thom \cite{KoTh16} based on a new result on the diameter of the Cayley graph of $S_k$ \cite{HeSe14}. However, the method of finding short identities in $S_k$ uses chains of iterated commutators and thus cannot be translated to produce short positive identities. So the problem of the existence of short positive identities remains open. Here we present some series of such identities, showing that $\Sepp(n)\ge \frac32\log n+o(\log n)$. Besides this, we present the results of computer-assisted studies for small $k$, providing, in particular, some exact values for the functions $\Sep$ and $\Sepp$.  

The rest of the paper consists of two sections. In Section~\ref{s:semi} we present our results on $\Sep$ and the identities in $T_k$, while in Section~\ref{s:grou} we consider $\Sepp$ and positive identities in $S_k$, together with the connection between $\Sep$ and $\Sepp$.



\section{Identities in $T_k$} \label{s:semi}

An identity $(u,v)$ of a semigroup $T$ is \emph{reducible} if there is an identity $(u',v')$ of $T$ and a nonempty word $w$ such that either $u=wu', v=wv'$, or $u=u'w, v=v'w$; otherwise, the identity is said to be \emph{irreducible}. Since we are interested in short identities, we will consider only irreducible ones. As was already observed, the  shortest irreducible unary identity of any semigroup $T_k$ is  identity \eqref{unary}. The following easy fact is well known; a proof can be found in \cite{DESW11}.

\begin{fact} \label{binary}
For any pair of non-unary words $(u,v)$ such that $u\equiv_k v$ there is a pair $(u',v')$ of \emph{binary} words such that $|u'|=|u|$, $|v'|=|v|$, and $u'\equiv_k v'$.
\end{fact}

Hence, in the quest for short non-unary identities in $T_k$ we restrict ourselves to identities and dfa's over the binary alphabet $\{x,y\}$. The following necessary conditions for an identity in $T_k$ are known from \cite{GoKo86,Rob89,DESW11}.

\begin{fact} \label{necessary}
If $u\equiv_k v$, then the words $u,v$ have (i) the same prefix of length $k{-}2$, (ii) the same suffix of length $k{-}1$, and (iii) the same set of factors of length $k{-}1$\footnote{This is related, but not equivalent, to the $(k{-}1)$-\emph{Abelian equivalence} of $u$ and $v$. The notion of $k$-Abelian equivalence is popular in modern combinatorics of words; see, e.g., \cite{KSZ13} and the references therein.}.
\end{fact}

We illustrate this fact with Fig.~\ref{f:nesessary}, showing the dfa's separating $u$ and $v$ in the case of violation of the conditions (i)--(iii).

\begin{figure}[!htb] 
\centerline{
\begin{tabular}{ll}
a)&\includegraphics[page=1]{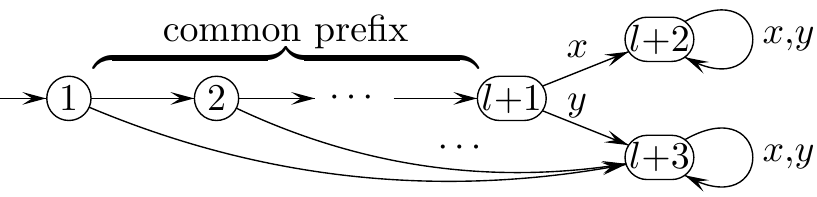}\\
b)&\includegraphics[page=2]{ShortIdentities-pics.pdf}\\
c)&\includegraphics[page=3]{ShortIdentities-pics.pdf}
\end{tabular}
}
\caption{\footnotesize Separation by prefixes, suffixes, and factors: (a) such a dfa with with $l{+}3$ states separates two words having the common prefix of length exactly $l$; (b) this example of the \emph{Aho-Corasick automaton} finishes its work in the rightmost state if and only if the input word has the suffix $xyyxxy$; such a dfa can be built for any suffix; (c) this variation of the previous automaton reaches the rightmost state if and only if the input word contains the factor $xyyxxy$; again, such a dfa can be built for any factor.}
\label{f:nesessary}
\end{figure}

Recall that, given a word $w\in\Sigma^*$ and a dfa $\cA$, $w$ can be viewed as a transformation of the set of states of $\cA$. The digraph of this transformation has one or more cycles (see an example in Fig.~\ref{f:trans}). Each such sycle is referred to as a \emph{$w$-cycle}.

\begin{figure}[!htb] 
\centerline{\includegraphics[page=4]{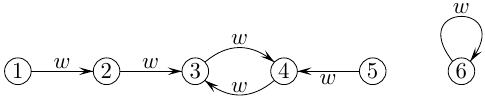}}
\caption{\footnotesize An example of the transformation of the set of states by a word.}
\label{f:trans}
\end{figure}

An identity $(u,v)$ is \emph{uniform} if $|u|=|v|$. First consider non-uniform identities. 
\begin{proposition} \label{p:nonuniform}
A unique shortest binary non-uniform irreducible identity is
\begin{equation} \label{nonuniform}
x^{k-2}yx^{k-1}\equiv_k x^{k-2+\lcm(k)}yx^{k-1}
\end{equation}
\end{proposition}
\begin{proof}
First we use Fact~\ref{sep_ide} to check that \eqref{nonuniform} is an identity. Consider any binary dfa $\cA=(\{x,y\},Q,\delta,s)$, $|Q|=k$, and prove that $\cA$ does not separate the parts of  \eqref{nonuniform}. To separate them, $\cA$ should separate $x^{k-2}$ from $x^{k-2+\lcm(k)}$. If the state $s.x^{k-2}\in Q$ belongs to an $x$-cycle, no separation is possible, because the length of this cycle divides $\lcm(k)$. Hence $s.x^{k-2}$ does not belong to an $x$-cycle. Then $Q=\{s, s.x, \ldots, s.x^{k-1}\}$ and the only $x$-cycle is the loop on the state $s.x^{k-1}$. Therefore, $x^{k-1}$ acts on $Q$ as a constant, implying that $\cA$ is unable to separate the parts of \eqref{nonuniform}.

Now assume that $u\equiv_k v$ and $|u|<|v|\le \lcm(k)+2k-2$ (this number is the length of identity \eqref{nonuniform}). By Fact~\ref{sep_ide}, $\Sep(u,v)>k$. Let $|u|_x=l$, $|v|_x=l+m$, and w.l.o.g. $m>0$. 
If $m$ is not divisible by $\lcm(k)$, then some $i\le k$ does not divide $m$. In this case $u$ and $v$ are separated by the $i$-state dfa in which $y$ is the identity map and $x$ is a cyclic permutation. Therefore the restriction on the length of $v$ implies $m=\lcm(k)$. By the same argument, $|u|_y=|v|_y$. So $|v|-|u|=\lcm(k)$, as well as in \eqref{nonuniform}. In addition, $u$ and $v$ satisfy the conditions (i)--(iii) of Fact~\ref{necessary}. Let $|u|<2k-2$. Then $u$ is completely covered by its prefix from (i) and its suffix from (ii). Then all $y$'s in $v$ occur in this prefix and/or suffix. Hence $v$ contains $x^{k-1}$; by (iii), so does $u$. Let $u=zx^{k-1}w$  for some words $z,w$. Since $u$ is short, $z$ (resp., $v$) is a part of the common prefix (resp., suffix) of $u$ and $v$. So $v=zx^{k-1+\lcm(k)}w$. But this means that the identity $u\equiv_k v$ is reducible to \eqref{unary}. This contradiction proves the assumption $|u|<2k-2$ false.

Finally, let $|u|=2k-2$, $z=u[1..k{-}2]$, $a=u[k{-}1]$, $w=u[k..2k{-}2]$. Then $u=zaw$ and $v=zv'w$ for some word $v'$ of length $\lcm(k)+1$. The equality $|u|_y=|v|_y$ implies that $v'$ contains exactly one $y$ if $a=y$ and $v'=x^{\lcm(k)+1}$ otherwise. Either way, $v'$ is long enough to contain the factor $x^{k-1}$, so $u$ contains it as well. If this factor is not a suffix of $u$, then $u\equiv_k v$ is reducible to \eqref{unary} as in the previous paragraph. Hence $w=x^{k-1}$. If $v$ has the prefix $za$, then this prefix contains all $y$'s in $v$; so $u=zax^{k-1}$, $v=zax^{\lcm(k)+k-1}$, and again our identity is reducible to \eqref{unary}. Therefore $u$ begins with $zy$ and $v$ begins with $zx$ (the opposite case is impossible since $|u|_y=|v|_y$). Note that $zy$ is a factor of $v$ by Fact~\ref{necessary}(iii). Since $v$ has a unique $y$ outside its prefix $z$ (it is in $v'$), this $y$ is preceded by $z$. So $v$ has two occurrences of $z$, and they together contain the same number of $y$'s as the prefix $z$ of $u$. This is possible only if $z=x^{k{-}2}$. Thus, each of $u=x^{k-2}yx^{k-1}$ and $v$ contain a single occurrence of $y$; say, $v[l]=y$. We have $l>k-1$, because $v$ begins with $zx=x^{k-1}$. If $l$ and $k-1$ are distinct modulo $i$ for some $i\le k$, then a dfa separating $u$ and $v$ is easy to construct: an $x$-cycle of length $i$ contains the initial vertex, and the $y$-edges from $s.x^{k-1}$ and $s.x^l$ lead to the same vertex of this cycle, so that the remaining $x$'s will be read to different vertices. Therefore, $l=k-1+\lcm(k)$, implying that the identity $u\equiv_k v$ coincides with \eqref{nonuniform}.
\end{proof}

Next we switch to uniform identities. An identity $(u,v)$ is \emph{balanced} if $|u|_a=|v|_a$ for any letter $a$.

\begin{proposition} \label{p:unbal}
A unique shortest binary uniform unbalanced identity is
\begin{equation} \label{unbal}
x^{k-1+\lcm(k)}y^{k-1}\equiv_k x^{k-1}y^{k-1+\lcm(k)}
\end{equation}
\end{proposition}
\begin{proof}
Since \eqref{unbal} is obtained by multiplying two copies of \eqref{unary}, it is obviously an identity. Now consider any uniform unbalanced identity $u\equiv_k v$ of length at most $\lcm(k)+2k-2$, which is the length of  \eqref{unbal}. Similar to the proof of Proposition~\ref{p:nonuniform}, we obtain that $|u|_x>|v|_x$ implies $|u|_x=|v|_x+\lcm(k)$ and $|v|_y=|u|_y+\lcm(k)$. Let $u=zu'w$, $v=zv'w$, where $z$ (resp. $w$) is the longest common prefix (resp., suffix) of $u$ and $v$. By Fact~\ref{necessary} we have $|z|\ge k-2$, $|w|\ge k-1$, and thus $|u'|\le \lcm(k)+1$. If $|u'|=\lcm(k)+1$, we can assume $u'=x^{\lcm(k)+1}$, $v'=y^ixy^j$, where $i,j>0$ (if $u'$ contains fewer $x$'s, then $v'=y^{\lcm(k)+1}$, so we get a symmetric case). Then $x^{k-1}$ is a factor of $v$ by Fact~\ref{necessary}, implying $w=x^{k-1}$. Now all factors of $u$ of length $k-1$ end with $x$, which is not the case for $v$; again by Fact~\ref{necessary}, $u$ and $v$ cannot form an identity. Hence, $|u'|\le \lcm(k)$. So we have $u'=x^{\lcm(k)}$, $v'=y^{\lcm(k)}$. Since $x^{k-1}$ is a factor of $v$, $y^{k-1}$ is a factor of $u$, we immediately get the identity \eqref{unbal} up to renaming the letters.
\end{proof}

\begin{proposition}
Every $T_k$ satisfies the binary uniform balanced identity
\begin{equation} \label{bal}
x^{k-2+\lcm(k)}yx^{k-1}\equiv_k x^{k-2}yx^{k-1+\lcm(k)}
\end{equation}
\end{proposition}
\begin{proof}
The same argument as in Proposition~\ref{p:nonuniform} works: for any dfa with $k$ states either $s.x^{k-2}=s.x^{k-2+\lcm(k)}$ or $x^{k-1}$ is a constant map.
\end{proof}

The summary of the proved statements is as follows: the shortest non-unary unbalanced identities in the semigroup $T_k$ have exactly the same length $\lcm(k)+2k-2$ as some binary balanced identity, and are slightly longer than the unary identity of this semigroup. The question is whether there exist shorter balanced binary identities. 

\begin{remark} \label{search4}
An exhaustive computer search reveals that identities \eqref{bal} are the shortest binary identities in the semigroups $T_k$ for $k\le 4$. For $k=5$, such a search is beyond capabilities of any computer. However, below we show that $T_5$ does have a shorter identity as well as infinitely many other semigroups $T_k$. 
\end{remark}
\begin{theorem}\label{th:sem}
Semigroup $T_k$ satisfies the following identity of length $2\lcm(k-1)+6(k-1)$:
\begin{equation} \label{shortsem}
(xy)^{k-2+\lcm(k-1)}(yx)^k(xy)^{k-1} \equiv_k (xy)^{k-2}(yx)^k(xy)^{k-1+\lcm(k-1)}
\end{equation}
\end{theorem}
\begin{corollary}
If $k\ge 5$ is either a prime or an odd prime power, the semigroup $T_k$ satisfies an identity which is shorter than the unary identity \eqref{unary}.
\end{corollary}
\begin{proof}[Proof of Theorem~\ref{th:sem}]
Let us take a dfa $\cA$ and consider the transformation $xy$ in it. If the state $s.(xy)^{k-2}$ does not belong to any $(xy)$-cycle, then we see, similar to Proposition~\ref{p:nonuniform}, that $(xy)^{k-1}$ is a constant map. So in this case $\cA$ does not separate the sides of \eqref{shortsem}. Assume that $s.(xy)^{k-2}$ belongs to a $(xy)$-cycle of length $m$. If $m<k$, then all $(xy)$-cycles in $\cA$ have length $<k$. Since $q.(xy)^{k-1}$ belongs to some $(xy)$-cycle for any state $q$ and the lengths of all $(xy)$-cycles divide $\lcm(k{-}1)$, both sides of \eqref{shortsem} move $s$ to the same state. Finally, let $m=k$. Then $xy$ is a permutation (namely, a cycle of length $k$), and $(xy)^k=1$. Hence $x,y$ and $yx$ are permutations, and clearly $(yx)^k=1$. Deleting $(yx)^k$ from both sides of \eqref{shortsem}, we get a graphical equality, so once again we see that $\cA$ is not separating. 
\end{proof}

\begin{conjecture} \label{c:short5}
Identity \eqref{shortsem} for $k=5$ is the shortest identity of $T_5$.
\end{conjecture}

This conjecture is partially verified by the computations described in the next section.

\section{Positive Identities in $S_k$} \label{s:grou}

The symmetric group $S_k$ satisfies the positive identity $x^{\lcm(k)}=1$ and its binary counterpart $x^{\lcm(k)}=y^{\lcm(k)}$. By the same argument, as the one used in Propositions~\ref{p:nonuniform} and~\ref{p:unbal}, these are the shortest unbalanced identities in $S_k$, so all shorter positive identities are balanced. It is known that the shortest positive identity in $S_3$ is $x^2y^2=y^2x^2$ (folklore). The shortest such identity in $S_4$ has length 11: $x^6y^2xy^2=y^2xy^2x^6$ \cite{DESW11}. We ran a computer search for the positive identities in $S_5$. Using an optimized search based on hash functions, we checked all balanced pairs $(u,v)$ of length at most 33, arriving at the following result.

\begin{proposition} \label{search5}
The shortest positive identities in $S_5$ have length 32. Up to symmetry, there are two such identities of length 32:
\begin{subequations}\label{32}
\begin{align}
(xy)(xyyx)^3(yxxy)^2(yx)(yxxy)^2&=(yxxy)^2(xy)(yxxy)^2(xyyx)^3(yx)\label{32:1}\\
(xy)^4(yx)^5(xy)^6(yx)&=(yx)(xy)^6(yx)^5(xy)^4\label{32:2}
\end{align}
\end{subequations}
Also, $S_5$ satisfies no irreducible positive identity of length 33.
\end{proposition}

Further, we checked the identities \eqref{32} in $S_6$.

\begin{proposition}
A unique, up to symmetry, shortest positive identity of $S_6$ is \eqref{32:2}. 
\end{proposition}

Naturally enough, \eqref{32:2} is not an identity in $S_7$: these words are separated by a dfa in which $xy$ and $yx$ are different cycles of length 7. Hence, the function $\Sepp(n)$ never takes the value 6:

\begin{proposition}
One has $\Sepp(1)=2,\Sepp(2)=\Sepp(3)=3,\Sepp(4)=\ldots=\Sepp(10)=4,\Sepp(11)=\ldots=\Sepp(31)=5,\Sepp(32)=\Sepp(33)=7$. 
\end{proposition}

\begin{proposition}
One has $\Sep(1)=\Sep(2)=2,\Sep(3)=\ldots=\Sep(7)=3,\Sep(8)=\ldots=\Sep(14)=4,\Sep(15)=\ldots=\Sep(40)=5,\Sep(48)>5$. 
\end{proposition}
\begin{proof}
Since identity \eqref{bal} is longer than \eqref{unary}, Remark~\ref{search4} implies the values of $\Sep$ up to $n=14$ and the fact that $\Sep(15)>4$.

Let $u\equiv_5 v$. Then $u\cong_5 v$ and, by Fact~\ref{necessary}, $u$ and $v$ have a common prefix of length 3 and a common suffix of length 4. A direct check shows that the identities \eqref{32} cannot produce an identity in $T_5$ of length 39 or 40, so $\Sep(n)$ equals 5 for $n=15,\ldots,40$ by Proposition~\ref{search5}. The last result follows from  Theorem~\ref{th:sem}.
\end{proof}

Identities \eqref{32} possess interesting properties. First, in both cases $u,v\in\{xy,yx\}^*$. Second, \eqref{32:1} is a palindrome ($v$ is the reversal of $u$), while \eqref{32:2} is a palindrome if considered over $\{xy,yx\}$. Having observed this, we performed a further search for identities in $S_5$ up to length 40, examining all pairs $(u,v)$ such that either $u,v\in\{xy,yx\}^*$ or $v$ is the reversal of $u$. The search revealed eight more identities; they are presented in Table~\ref{idS5}. Note that some of them hold in $S_6$ but none holds in $S_7$.

\begin{table}[!htb]
\caption{More short positive identities in $S_5$.}
\label{idS5}
\tabcolsep=2pt
\begin{tabular}{|c|c|c|c|c|}
\hline
no.&$|u|$&Identity&Type&Hold in $S_6$?\\
\hline
1&34&$(xy)^{12}(yx)^5 = (yx)^5(xy)^{12}$&$\{xy,yx\}$-pal.&Yes\\ 
2&38&$(xy)^4(yx)^5(xy)^6(yx)(xy)^2(yx) = (yx)(xy)^2(yx)(xy)^6(yx)^5(xy)^4$&$\{xy,yx\}$-pal.&Yes\\
3&38&$(xy)^2(yx)^3(xyyx)^2(xy)^2(yxxy)^2(xyyx)^2 = $&&\\
&&$(yxxy)^2(xyyx)^2(xy)^2(yxxy)^2(yx)^3(xy)^2$&$\{xy,yx\}$-pal.&No\\
4&39&$(x^2y^2)^2y(x^2y^2)^4x^2y(x^2y^2)^2x^2y = yx^2(y^2x^2)^2yx^2(y^2x^2)^4y(y^2x^2)^2$&palindrome&No\\
5&39&$(x^2y^2)^3y(x^2y^2)^4x^2y(x^2y^2)x^2y = yx^2(y^2x^2)yx^2(y^2x^2)^4y(y^2x^2)^3$&palindrome&No\\
6&40&$(xyyx)^3(yxxy)^5(xyyx)^2 = (yxxy)^2(xyyx)^5(yxxy)^3$&$\{xy,yx\}$-pal.&No\\
7&40&$(xy)^6(yx)^{10}(xy)^4 = (yx)^4(xy)^{10}(yx)^6$&palindrome&Yes\\
8&40&$(x^2y^2)^3(y^2x^2)^5(x^2y^2)^2 = (y^2x^2)^2(x^2y^2)^5(y^2x^2)^3$&palindrome&No\\
\hline
\end{tabular}
\end{table}

Note that if $zuw\equiv_k zvw$, where $z$ (resp., $w$) is the longest common prefix (resp., suffix) of both sides, then $u\cong_k v$. So, the search for the identities in $T_5$ can be performed by iterating over the identities of $S_5$, using an exhaustive search for the candidates for $z$ and $w$. Such a search, based on the identities listed in \eqref{32} and Table~\ref{idS5}, gave us exactly one identity of $T_5$, namely, the identity \eqref{shortsem} for $k=5$, that has length 48. The result of this search supports  Conjecture~\ref{c:short5}.

\medskip
The analisys of the identities listed in \eqref{32} and Table~\ref{idS5} results in finding some general classes of identities in $S_k$. The simplest class, described in the following proposition, allows us to move up the lower bound on the function $\Sepp$ by a multiplicative constant.

\begin{proposition} \label{p:ab}
Let $a,b$ be such that the order of any element of $S_k$ divides either $a$ or $b$. Then 
\begin{equation} \label{e:ab}
(xy)^a(yx)^b\cong_k (yx)^b(xy)^a\,.
\end{equation}
\end{proposition}
\begin{proof}
For any $x,y\in S_k$ the elements $(xy)$ and $(yx)$ have the same order. Then by the choice of $a,b$ either $(xy)^a=1$ or $(yx)^b=1$, implying the result.
\end{proof}

\begin{theorem} \label{t:ab}
The symmetric group $S_k$ satisfies a positive identity \eqref{e:ab} of length $e^{\frac 23 k+O(\frac{k}{\log k})}$. 
\end{theorem}

\begin{corollary} \label{seplow}
$\Sepp(n)\ge \frac 32 \log n + O\big(\frac{\log n}{\log\log n}\big)$.
\end{corollary}

\begin{proof}[Proof of Theorem~\ref{t:ab}]
Take a number $\alpha$, $0<\alpha<1$. Let $m=\lfloor \alpha k\rfloor$ and $P(m)$ be the product of all primes and prime powers from the range $\{m+1,\ldots, k\}$. Choose $a=\lcm(m)$, $b=\lcm(k-m)\cdot P(m)$, and apply Proposition~\ref{p:ab}. Indeed, the order of a permutation is the least common multiple of the length of its cycles; if a permutation has no cycle of length greater than $m$, than its order divides $a$; if such a cycle exists, than all other cycles are shorter than $k-m$, so the order divides $b$. Thus we get an identity of type \eqref{e:ab} with the $a$ and $b$ chosen\footnote{It is easy to see that one can take a smaller number as $b$, replacing $k{-}m$ with $k{-}m{-}1$ and the product of $\lcm$ and $P$ with their least common multiple. However, such an improvement does not change the asymptotics: its effect is covered by the $O$-term in the asymptotic formula.}. Since the length of this identity is $2(a+b)$, we want to find the value of $\alpha$ which delivers the minimum to $a{+}b$. Clearly, $\alpha\ge 1/2$, implying $m\ge k/2$. We use standard asymptotic formulas (see, e.g., \cite{BaSh96}) $\lcm(t)=e^{t+O(\frac t{\log t})}$ and $\pi(t)=\frac t{\log t}+ O(\frac t{\log^2 t})$, where $\pi(t)$ is the number of primes smaller than $t$. To estimate $P(m)$, we note that the product of $i$ factors equals their geometric mean taken to the $i$th power. Since all factors are between $m$ and $k$, their mean is $k/\beta$ for some $\beta$ between 1 and 2. To compute the number of factors, we can use the asymptotics for $\pi(m)$ (the number of prime powers smaller than $t$ is $O(\pi(\sqrt{t}))$ and thus does not affect the asymptotics). So we have
\begin{align*}
a=&e^{m+O(\frac m{\log m})}=e^{m+O(\frac k{\log k})},\\
b=&e^{k-m+O(\frac {k-m}{\log (k-m)})}\cdot \Big(\frac k{\beta}\Big)^{\frac k{\log k}-\frac m{\log m}+ O(\frac k{\log^2 k})}=e^{2k-2m+O(\frac k{\log k})}
\end{align*}
Thus the minimum of $a{+}b$ is reached at $\alpha=2/3$ so that $m=2k/3$, and this minimum is $e^{\frac 23 k+O(\frac{k}{\log k})}$, as required.
\end{proof}

A more involved class of equations is defined in the following proposition. The corresponding conditions can be easily extended to get identities with any even number of blocks of the form $(xy)^{a}$ and $(yx)^{b}$, but it is not clear if it is possible to build short identities of this type for any $k$.

\begin{proposition} \label{p:abcd}
Let $a,b,c,d$ be such that every order $q$ of an element of $S_k$ satisfies at least one the following conditions or their counterparts obtained by swapping $b$ with $c$, and $a$ with $d$: (i) $q$ divides both $a$ and $c$, (ii) $q$ divides both $a+c$ and $b$, (iii) $q$ divides $a$ and $b\equiv d \pmod q$. Then $S_k$ satisfies the identity
\begin{equation} \label{e:abcd}
(xy)^a(yx)^b(xy)^c(yx)^d\cong_k (yx)^d(xy)^c(yx)^b(xy)^a\,.
\end{equation}
\end{proposition}

\begin{proof}
We again use the fact that for any $x,y\in S_k$ the elements $(xy)$ and $(yx)$ have the same order. It is easy to see that each of the conditions (i)--(iii) forces some terms to vanish from both sides of \eqref{e:abcd} in a way that the remaining words are graphically equal. 
\end{proof}

We use Propositions~\ref{p:ab} and~\ref{p:abcd} to run further computer experiments; in Table~\ref{idSk} we present the parameters of the shortest identities of types \eqref{e:ab} and \eqref{e:abcd}, obtained by exhaustive search, and compare their lengths to the length $\lcm(k)$ of the unary identity. Note that the parameters $a$ and $b$ of the shortest identity of type \eqref{e:ab} in most cases are equal to those chosen by the rule described in the proof of Theorem~\ref{t:ab}. For example, for $k=23$ we have $a=\lcm(16)$, $b=\lcm(6)\cdot17\cdot19\cdot23$. So it looks probable that no other way of choosing the pair $(a,b)$ can improve the result of Theorem~\ref{t:ab}. The identities of type \eqref{e:abcd} for small $k$ are shorter than the identities of type \eqref{e:ab}, but it is unclear whether this is true for all $k$.

\begin{table}[!htb]
\caption{Parameters of the shortest positive identities of types \eqref{e:ab},\eqref{e:abcd}.}
\label{idSk}
\arraycolsep=2pt
$$
\begin{array}{|c||c|c|c|c|c||c|c|c||c|}
\hline
&\multicolumn{5}{|c||}{\text{Identities of type \eqref{e:abcd}}}&\multicolumn{3}{|c||}{\text{Identities of type \eqref{e:ab}}}&\\
\hline
k&a&b&c&d&\text{Len}&a&b&\text{Len}&\lcm(k)\\
\hline
5,6&1&6&5&4&32&12&5&34&60\\
7&2&14&12&10&76&60&7&134&420\\
8&23&60&7&24&228&60&56&232&840\\
9&18&60&42&24&288&180&56&472&2520\\
10&18&60&42&24&288&120&126&492&2520\\
11&48&180&132&84&888&840&198&2076&27720\\
12&24&222&420&198&1728&840&198&2076&27720\\
13&&&&&&2520&286&5612&360360\\
14&&&&&&2520&858&6756&360360\\
15&&&&&&2520&1716&8472&360360\\
16&&&&&&5040&8580&27240&720720\\
17&&&&&&27720&10608&76656&12252240\\
18&&&&&&55440&13260&137400&12252240\\
19&&&&&&55440&251940&614760&232792560\\
20&&&&&&360360&15504&751728&232792560\\
21&&&&&&360360&77520&875760&232792560\\
22&&&&&&360360&77520&875760&232792560\\
23&&&&&&720720&445740&2332920&5354228880\\
\hline
\end{array}
$$
\end{table}

\section{Conclusion}

In this paper, we did the very first step in improving the lower bound on words separation (or, from the other point of view, improving the upper bound on the shortest identity in full transformation semigroups and the shortest positive identity in symmetric groups). Apart from the experimentally obtained values of the separation functions $\Sep$ and $\Sepp$ for small arguments, we obtained two asymptotic results:
\begin{itemize}
\item the logarithmic lower bound for $\Sep(n)$ is improved by an additive sublogarithmic term for infinitely many values of $n$;
\item the logarithmic lower bound for $\Sepp(n)$ is improved by a factor of $3/2$.
\end{itemize}
The obvious next step should be an attempt to improve the function $\Sep$ by some factor and prove a superlogarithmic lower bound for $\Sepp$. Our general impression is that both such improvements are possible. On the other hand, we are not so optimistic about the existence of a superlogarithmic lower bound for $\Sep$.

\bibliographystyle{plain} 
\bibliography{my_bib} 

\end{document}